\documentclass{amsart}
\usepackage{latexsym,amsmath,amssymb}
\newtheorem{thm}{Theorem}[section]
\newtheorem{cor}[thm]{Corollary}
\newtheorem{lem}[thm]{Lemma}
\newtheorem{exam}[thm]{Example}
\newtheorem{pro}[thm]{Proposition}
\theoremstyle{definition}
\theoremstyle{remark}
\newtheorem{rem}[thm]{Remark}
\numberwithin{equation}{section}

\begin{document}

\title[On Properties of  MULTIPLICATION AND COMPOSITION OPERATORS]
{On Properties of MULTIPLICATION AND COMPOSITION OPERATORS between ORLICZ SPACES}

\author{\sc\bf Y. Estaremi, S. Maghsodi and I. Rahmani }
\address{\sc Y. Estaremi, S. Maghsodi and I. Rahmani }
\email{estaremi@gmail.com}
\email{mathrahmani@znu.ac.ir}

\address{ Department of Mathematics, Payame Noor University (PNU), P. O. Box: 19395-3697, Tehran- Iran}
\address{Department of Mathematics, University of Zanjan, Zanjan, Iran}
\address{Department of Mathematics,University of Zanjan, Zanjan, Iran}

\thanks{}

\thanks{}

\subjclass[2010]{47B47, 47B33}

\keywords{Multiplication operator, Composition operator, Closed range operator, Fredholm operator. }

\date{}

\dedicatory{}

\commby{}

\begin{abstract}
In this paper, we study bounded and closed range  multiplication and composition operators between two different Orlicz spaces.

\noindent {}
\end{abstract}

\maketitle

\section{ \sc\bf Preliminaries and Introduction }

 The continuous convex function
$\Phi:\mathbb{R}\rightarrow\mathbb{R}$ is called a Young function whenever
\begin{enumerate}
\item $\Phi(x)=0$ if and only if $x=0$.

\item $\Phi(x)=\Phi(-x)$.

\item $\lim_{x\rightarrow\infty}\frac{\Phi(x)}{x}=\infty$, $\lim_{x\rightarrow\infty}\Phi(x)=\infty$.

\end{enumerate}
 With each Young function $\Phi$ one can associate another
 convex function $\Psi:\mathbb{R}\rightarrow\mathbb{R}^{+}$ having similar properties, which is defined by
$$\Psi(y)=\sup\{x|y|-\Phi(x):x\geq0\}, \ \ y\in\mathbb{R}.$$
Then $\Psi$ is called the complementary Young function of $\Phi$.
A Young function $\Phi$ is said to satisfy the
$\bigtriangleup_{2}$ condition (globally) if $\Phi(2x)\leq
k\Phi(x), \ x\geq x_{0}\geq0 \ \  (x_{0}=0)$ for some constant
$k>0$. Also, $\Phi$ is said to satisfy the
$\bigtriangleup'(\bigtriangledown')$ condition, if $\exists c>0$
$(b>0)$ such that
$$\Phi(xy)\leq c\Phi(x)\Phi(y), \ \ \ x,y\geq x_{0}\geq 0$$
$$(\Phi(bxy)\geq \Phi(x)\Phi(y), \ \ \ x,y\geq y_{0}\geq 0).$$
If $x_{0}=0(y_{0}=0)$, then these conditions are said to hold
globally. If $\Phi\in \bigtriangleup'$, then $\Phi\in
\bigtriangleup_{2}$.\\

 Let $\Phi_{1}, \Phi_{2}$ be two Young
functions, then $\Phi_1$ is stronger than $\Phi_2$,
$\Phi_1\succ\Phi_2$ [or $\Phi_2\prec\Phi_1$] in symbols, if
$$\Phi_2(x)\leq\Phi_1(ax), \ \ \ x\geq x_0\geq0$$
for some $a\geq0$ and $x_0$, if $x_0=0$ then this condition is
said to hold globally.

Let $(X, \Sigma, \mu)$ be a $\sigma$-finite
complete measure space and $L^0(\Sigma)$ be the linear space of all equivalence classes of $\Sigma$-measurable
functions on $X$, that is, we identify any two functions that
are equal $\mu$-almost everywhere on $X$. The support of a
measurable function $f$ is defined as $S(f)=\{x\in X; f(x)\neq
0\}$. Let $\Phi$ is a Young function, then the set of
$\Sigma-$measurable functions $$L^{\Phi}(\Sigma)=\{f\in L^0(\Sigma):\exists k>0,
\int_{\Omega}\Phi(kf)d\mu<\infty\}$$ is a Banach space, with
respect to the norm
$\|f\|_{\Phi}=\inf\{k>0:\int_{\Omega}\Phi(\frac{f}{k})d\mu\leq1\}$.
$(L^{\Phi}(\Sigma), \|.\|_{\Phi})$ is called an Orlicz space
\cite{raor}.

For a measurable function $u\in L^0(\Sigma)$, the rule taking $u$ to $u.f$, is a linear
transformation on $L^0(\Sigma)$ and we denote this transformation
by $M_{u}$. In the case that $M_{u}$ is continuous, it is called
multiplication operator induced by $u$.

 Let $T:X\rightarrow X$
be a measurable transformation, that is, $T^{-1}(A)\in
\Sigma$ for any $A\in \Sigma$. If $\mu(T^{-1}(A))=0$ for all
$A\in \Sigma$ with $\mu(A)=0$, then $T$ is said to be
nonsingular. This condition means that the measure
$\mu\circ T^{-1}$, defined by
$\mu\circ T^{-1}(A)=\mu(T^{-1}(A))$ for $A\in \Sigma$,
is absolutely continuous with respect to the $\mu$ (it is usually
denoted $\mu\circ T^{-1}\ll \mu$). The Radon-Nikodym theorem
ensures the existence of a nonnegative locally integrable function
$f_0$ on $X$ such that,
$\mu\circ T^{-1}(A)=\int_{A}f_0d\mu$, $A\in \Sigma$. Any
 nonsingular measurable transformation $T$ induces a linear
operator (composition operator) $C_{T}$ from $L^0(\Sigma)$
into itself defined by$$C_{T}(f)(t)=f(T(t)) \ \ \ ;
t\in X, \ \ \ f\in L^0(\Sigma).$$  Here the
non-singularity of $T$ guarantees that the operator
$C_{T}$ is well defined as a mapping from $L^0(\Sigma)$ into
itself.

The composition and multiplication operators received considerable
attention over the past several decades especially on some
measurable function spaces such as $L^P$-spaces, Bergman spaces
and a few ones on Orlicz spaces, such that these operators played
an important role in the study of operators on Hilbert spaces. The multiplication and
weighted composition operators are studied on Orlicz spaces in
\cite{ks, skn}. Also, some
results on boundedness of composition operators on Orlicz spaces,
are obtained in \cite{c, ku} (see also \cite{ra}). In this paper we investigate composition and multiplication operators on Orlicz spaces by considering closed range, Fredholm and invertible ones.

\section{ \sc\bf  Bounded multiplication and composition operators }
In this section first we recall
that an $\Sigma$-atom of the measure $\mu$ is an element
$A\in\Sigma$ with $\mu(A)>0$ such that for each
$F\in\Sigma$, if $F\subseteq A$, then either $\mu(F)=0$ or
$\mu(F)=\mu(A)$. A measure space $(X,\Sigma,\mu)$ with no atoms is
called a non-atomic measure space \cite{z}.  It is well-known fact that every
$\sigma$-finite measure space $(X, \Sigma,\mu)$ can be
partitioned uniquely as $X=\left
(\bigcup_{n\in\mathbb{N}}A_n\right )\cup B$, where
$\{A_n\}_{n\in\mathbb{N}}$ is a countable collection of pairwise
disjoint $\Sigma$-atoms and $B$, being disjoint from each $A_n$,
is non-atomic. Also, in a $\sigma$-finite measure space all atoms have finite measure \cite{z}. Here we recall a fundamental Lemma that is easy to  prove.\\

\begin{lem}\label{l1} Let $\Phi_i$, $i=1,2,3$, be Young's functions for which
$$\Phi_3(xy)\leq\Phi_1(x)+\Phi_2(y), \ \ \ \ \ \ x,y\geq0,$$
If $f_i\in L^{\Phi_i}(\Sigma)$, $i=1,2$, where $(X,\Sigma,\mu)$ is a measure space, then
$$\|f_1f_2\|_{\Phi_3}\leq2\|f_1\|_{\Phi_1}\|f_2\|_{\Phi_2}.$$
\end{lem}
Here we give some necessary and sufficient conditions under which the multiplication operator $M_u$ is bounded between different Orlicz space.

\begin{thm}\label{t1}

Let $\Phi_1$ and $\Phi_2$ be Young's functions such that $\Phi_1, \Phi_2\in \bigtriangleup_2$ and $\Phi_2(xy)\leq\Phi_1(x)+\Phi_3(y), \  \ x\geq0,$
 for some Young's function $\Phi_3$. If  $u\in L^{\Phi_3}(\Sigma)$, then $u$ induces a bounded multiplication operator $M_u$ from $L^{\Phi_1}(\Sigma)$ into $L^{\Phi_2}(\Sigma)$.
\end{thm}

\begin{proof}
 Suppose that $u\in L^{\Phi_3}(\Sigma)$. Then by using Lemma \ref{l1}, for every $f\in L^{\Phi_1}(\Sigma)$ we get;
\begin{align*}
  ||M_u f||_{L^{\Phi_2}(\Sigma)}&=||u.f||_{L^{\Phi_2}(\Sigma)}\\
  &\leq2\|u\|_{L^{\Phi_3}(\Sigma)}\|f\|_{L^{\Phi_1}(\Sigma)}.
\end{align*}
Hence $M_u$ is a bounded multiplication operator from $L^{\Phi_1}(\Sigma)$ into $L^{\Phi_2}(\Sigma)$, and
\begin{center}
  $\|M_u\|\leq2\|u\|_{L^{\Phi_3}(\Sigma)}$.
\end{center}
\end{proof}

\begin{thm}\label{t2}   If $M_u$ is bounded from $L^{\Phi_1}(\Sigma)$ into $L^{\Phi_2}(\Sigma)$
and $\Phi_1\in \bigtriangleup'$. If $\Phi_3=\Psi_2\circ\Psi_1^{-1}$ is a Young's function, then $u\in L^{\Psi_3\circ\Psi_1}$, where $\Psi_i$'s are the complementary Young's functions of $\Phi_i$ for  $i=1,2,3$.

\end{thm}
\begin{proof}
Suppose that $M_u$  is bounded. Hence the adjoint operator $(M_u)^\ast:L^{\Psi_2}\longrightarrow L^{\Psi_1}$
is also bounded.
Since $\Phi_1\in\bigtriangleup'$, then $\Psi_1\in\nabla'$ and so there exists $b>0$ such that
for every $f\in L^{\Phi_3}(\Sigma)$ we have $\Psi_1^{-1}(f)\in L^{\Psi_2}(\Sigma)$ and so
 \begin{align*}\int_X\Psi_1(\overline{u})fd\mu&=\int_X\Psi_1(\overline{u})\Psi_1(\Psi_1^{-1}(f))d\mu\\
&\leq b\int_X\Psi_1(\overline{u}\Psi_1^{-1}(f))d\mu\\
&= b\int_X\Psi_1(M^\ast_u(\Psi_1^{-1}(f))d\mu<\infty.
\end{align*}
This means that $\Psi_1(\overline{u})\in\L^{\Psi_3}(\Sigma)$. In other words, $u\in L^{\Psi_3\circ\Psi_1}(\Sigma)$.
\end{proof}

For the underlying non-atomic measure spaces we have an important assertion as follows, that states there is not any bounded multiplication and composition operator from $L^{\Phi_1}(\Sigma)$ to $L^{\Phi_2}(\Sigma)$ when $\Phi_2\nless\Phi_1$.
\begin{pro}\label{p0} Let $\Phi_2\nless\Phi_1$ and $(X,\Sigma,\mu)$ be a non-atomic measure space, then there is
 no non-zero bounded operator $M_{u,T}=M_uC_T$ from $L^{\Phi_1}(\Sigma)$ to $L^{\Phi_2}(\Sigma)$.
\end{pro}

\begin{proof} Suppose on the contrary, let $M_{u,T}$ be a non-zero bounded liner operator. Let
$$E_n=\{x\in X:|u(x)|>\frac{1}{n}\}\bigcap\{x\in X:|f_0(x)|>\frac{1}{n}\}.$$
Then $\{E_n\}_{n\in \mathbb{N}}$ is an increasing sequence of measurable sets. Since $M_{u,T}$ is non-zero, then $\mu(E_n)>0$ for some $n\in N$. Suppose $F\subset E=\cup_n E_n$ with $\mu(F)<\infty$. Since $\Phi_2\nless\Phi_1$, then there exists a sequence of positive numbers $\{y_n\}$ such that $y_n\uparrow\infty$ and $\Phi_2(y_n)>\Phi_1(2^nn^3y_n)$. Since $\mu$ is non-atomic, we can find a disjoint sequence $\{F_n\}$ of measurable subsets of $F$ such that $F_n\subseteq E_n$ and
$\mu(F_n)=\frac{\Phi_1(y_1)\mu(F)}{2^n\Phi_1(n^3y_n)}$.
 Let $\displaystyle f=\sum^\infty_{n=1}b_n\chi_{F_n}$, where $b_n=n^2y_n$, then for $n_0>\alpha$ we have
 \begin{align*}
\int_X\Phi_1(\alpha f)d\mu&=\displaystyle\sum^{\infty}_{n=1}\int_X\Phi_1(\alpha b_n)\chi_{F_n}\\
&=\sum^{n_0}_{n=1}\Phi_1(\alpha.b_n)\mu(F_n)+\sum\limits_{n\geq n_0}\Phi_1(\alpha.b_n)\mu(F_n)\\
&=\sum^{n_0}_{n=1}\Phi_1(\alpha.b_n)\mu(F_n)+\sum_{n\geq n_0}\frac{\Phi_1(\alpha.b_n)\phi_1( y_1)}{2^n\phi_1(n^3 y_n)}\\
&\leq\sum^{n_0}_{n=1}\Phi_1(\alpha.b_n)\mu(F_n)+\mu(F)\sum_{n\geq n_0}\frac{\Phi_1(n^3y_n)\phi_1( y_1)}{2^n\phi_1(n^3 y_n)}<\infty.
\end{align*}
This implies that $f\in L^{\Phi_1}(\Sigma)$. But for $m_0>0$ for which $\frac{1}{m_0}<\alpha$, we have
\begin{align*}
\int_X\Phi_2(\alpha M_{u,T}f)d\mu&=\int_Xf_0(x)\Phi_2(\alpha u.f)d\mu\\
&\geq\sum\limits_{n\geq m_0}\int_{F_n}f_0(x)\Phi_2(\alpha u.b_n)d\mu\\
&\geq\sum\limits_{n\geq m_0}\int_{F_n}\frac{1}{n}\Phi_2( \frac{1}{n^2}.b_n)d\mu\\
&\geq\sum\limits_{n\geq m_0}\int_{F_n}\frac{1}{n}\Phi_2(y_n)d\mu\\
&\geq\sum_{n\geq m_0}\frac{1}{n}\Phi_1(2^nn^3 y_n)\mu(F_n)\\
&\geq\mu(F)\sum_{n\geq n_0}\frac{1}{n}\Phi_2(y_1)=\infty.
\end{align*}
Which contradicts boundedness of $M_{u,T}$.
\end{proof}

\begin{lem}\label{l2} Let $\Phi_i$, $i=1,2,3$, be Young's functions for which
$$\Phi_2(xy)\leq\Phi_1(x)+\Phi_3(y),\quad x,y\geq0,$$

then $\Phi_1\nless\Phi_2$.
\end{lem}

\begin{proof} It is easy to get that  $\Phi_1^{-1}(x)\Phi_3^{-1}(x)\leq\Phi_2^{-1}(2x)\leq2\Phi_2^{-1}(x)$, for all $x\geq0$,.
Suppose on the contrary, hence there exists $\delta>0$ and $N>0$ such that
 $$\Phi_1(x)<\Phi_2(\delta x), \ \ \ \ \ \forall x\geq N.$$ Thus
 we have $\Phi_2^{-1}(\Phi_1(x))<\delta x$ and so $\Phi_1^{-1}(\Phi_1(x))\Phi_3^{-1}(\Phi_1(x))<2\delta x$, $\forall x\geq N$.
This implies that $\Phi_1(x)<\Phi_3(2\delta)$, for all $x\geq N$. This is a contradiction.
\end{proof}

The next Proposition is a main tools that we use in our investigation. 

\begin{pro}\label{p1} Suppose that $\Phi_2\nless\Phi_1$ and $\Phi_2\in \vartriangle_2$. If $E$ is a
non-atomic measurable set with $\mu(E)>0$, then there exists $f\in L^{\Phi_1}(X)$ such that $f \notin L^{\Phi_2}(E)$.
\end{pro}

\begin{proof}
Suppose that $F\subset E$ and $\alpha=\mu(F)<\infty$. Since $\Phi_2\nless\Phi_1$, then we can find a sequence $\{x_n\}$ in $X$ such that $x_n\uparrow\infty$ with $\Phi_2(|x_n|)>\Phi_1(n|x_n|)$. Let $n_0\in \mathbb{N}$ such that $\alpha>\displaystyle\sum_{n\geq n_0}\frac{1}{n^2}$ and $\Phi_1(x_n)\geq1$ for all $n\geq n_0$. Since $\mu$ is non-atomic, then there exists a measurable set $F_0\subset F$ such that
$\mu(F_0)= \sum_{n\geq n_0}\frac{1}{n^2}$. Similarly we can find a set $F_1\in \Sigma, F_1\subset F_0$ such that $\mu(F_1)=n_0^{-2}$. Since
$\mu(F_0-F_1)>0$, we can again find $F_2\in \Sigma, F_2\subset F_0- F_1$ such that $\mu(F_2)=(n_0+1)^{-2}$. Repeating the process, we find disjoint sets $F_n\in \Sigma, F_n\subset F_{n-2}- F_{n-1}$ such that $\mu(F_n)=(n_0+n+1)^{-2}$. Let $E_k\subset F_k, E_k\in \Sigma$, be chosen such that $\mu(E_k)=\frac{\mu(F_k)}{\phi_1(|x_k|)}$. If we take $\displaystyle f=\sum^\infty_{n=1}x_n\chi_{E_n}$, then we have;
\begin{align*}
\int_X\Phi_1(|f|)d\mu&=\displaystyle\sum^{\infty}_{n=1}\int_X\Phi_1(|x_n|)\chi_{E_n}\\
&=\sum^{n_0}_{n=1}\Phi_1(|x_n|)\mu(E_n)+\sum\limits_{n\geq n_0}\Phi_1(| x_n|)\frac{\mu(F_n)}{\Phi_1(|x_n|)}<\infty.
\end{align*}
This means that $f\in L^{\Phi_1}(X)$. Also we have
\begin{align*}
\int_E\Phi_2(|f|)d\mu&=\displaystyle\sum^{\infty}_{n=1}\Phi_2(|x_n|)\mu(E_n)\\
&>\sum^{\infty}_{n>n_0}n\Phi_1(|x_n|)\mu(E_n)\\
&=\sum^{\infty}_{n>n_0}n\mu(F_n)\\
&=\sum^{\infty}_{n>n_0}\frac{1}{n}=\infty.
\end{align*}
This shows that $f\notin L^{\Phi_2}(E)$.
\end{proof}
Now we provide some necessary and sufficient conditions for boundedness of multiplication operators between different Orlicz spaces.
\begin{thm}\label{t3}  Let $u\in L^0(\Sigma)$, $\Phi_1, \Phi_2,\Phi_3$ be Young's functions such that $\Phi_1(xy)\leq\Phi_2(x)+\Phi_3(y)$. If $u$ induces a bounded multiplication operator $M_u:L^{\Phi_1}(\Sigma)\rightarrow L^{\Phi_2}(\Sigma)$, then
\begin{enumerate}

\item[(i)] $u(x)=0$ for $\mu$-almost all $x\in B$.

\item[(ii)] $\sup\limits_{n\in N}|u(A_n)|.\Phi_3^{-1}(\frac{1}{\mu(A_n)})<\infty$.
\end{enumerate}
\end{thm}

\begin{proof}
  Suppose that $M_u$ is bounded. First we prove (i). If $\mu\{x\in B;|u(x)|\neq0\}>0$, then there exists a positive constant $\delta$ such that $\mu\{x\in B;|u(x)|>\delta\}>0$. Put $E=\{x\in B;|u(x)|>\delta\}$. Since $\mu(E)>0$ and $E$ is non-atomic, then by Lemma \ref{l2} and Proposition \ref{p1}, there exists $f\in L^{\Phi_1}(\Sigma)$ such that $f\notin L^{\Phi_2}(E)$ and so
$$\infty=\int_E\Phi_2(\frac{\delta f(x)}{\parallel M_uf\parallel_{\Phi_2}})d\mu\leq\int_X\Phi_2(\frac{u(x). f(x)}{\parallel M_uf\parallel_{\Phi_2}})d\mu\leq1,$$
which is a contraction. Thus (i) holds. Now we prove (ii). For any $n\in N$, put $f_n=\Phi_1^{-1}(\frac{1}{\mu(A_n)})\chi_{A_n}$. It is clear that $f_n\in L^{\Phi_1}(\Sigma)$ and $\parallel f_n\parallel_{\Phi_1}=1$. So we have
\begin{align*}
1&\geq\int_X\Phi_2(\frac{u(x). f_n(x)}{\parallel M_uf_n\parallel_{\Phi_2}})d\mu\\
&=\int_{A_n}\Phi_2(\frac{u(x).\Phi_1^{-1}(\frac{1}{\mu(A_n)})}{\parallel M_uf_n\parallel_{\Phi_2}})d\mu\\
&=\Phi_2(\frac{u(A_n).\Phi_1^{-1}(\frac{1}{\mu(A_n)})}{\parallel M_uf_n\parallel_{\Phi_2}})\mu(A_n).
\end{align*}
Therefore $\frac{u(A_n).\Phi_1^{-1}(\frac{1}{\mu(A_n)})}{\parallel M_uf_n\parallel_{\Phi_2}}\leq\Phi_2^{-1}(\frac{1}{\mu(A_n)})$
and consequently by the proof of Lemma \ref{l2}
\begin{align*}
M=\sup_n u(A_n).\Phi_3^{-1}(\frac{1}{\mu(A_n)})\\
&\leq2\parallel M_uf_n\parallel_{\Phi_2}\\
&\leq2\|M_u\|<\infty.
\end{align*}
This completes the proof.
\end{proof}

\begin{cor} Under the assumptions of Theorem \ref{t3}, if $(X, \Sigma, \mu)$ is a non-atomic measure space, then the multiplication operator $M_u$ from $L^{\Phi_1}(\Sigma)$ into $L^{\Phi_2}(\Sigma)$ is bounded if and only if $M_u=0$.
\end{cor}

\begin{thm}\label{t44}
Let $u\in L^0(\Sigma)$, $\Phi_1, \Phi_2,\Phi_3$ be Young's functions such that  $\Phi_1,\Phi_2\in\triangle'$ and $\Phi_2\circ\Phi^{-1}_1$ be a Young's function. Then $u$ induces a bounded multiplication operator $M_u:L^{\Phi_1}(\Sigma)\rightarrow L^{\Phi_2}(\Sigma)$,  if
\begin{enumerate}

\item[(i)] $u(x)=0$ for $\mu$-almost all $x\in B$.

\item[(ii)] $\sup\limits_{n\in N}\Phi_2[\frac{u(A_n)}{\Phi_1^{-1}(\mu(A_n))}]\mu(A_n)<\infty$.
\end{enumerate}
\end{thm}

\begin{proof}
 Suppose that (i) and (ii) hold. Put $\sup\limits_{n\in N}\Phi_2[\frac{u(A_n)}{\Phi_1^{-1}(\mu(A_n))}]\mu(A_n)=M $. Then for each $f\in L^{\Phi_1}(X)$, we have
\begin{align*}
\int_X\Phi_2(M_uf)d\mu&=\int_{x\in B}\Phi_2(u(x).f(x))d\mu+\int_{x\in \bigcup A_n}\Phi_2(u(x).f(x))d\mu\\
&=\sum\limits_{n\in N}\int_{x\in A_n}\Phi_2(u(x).f(x))d\mu\\
&=\sum\limits_{n\in N}\Phi_2(u(A_n).f(A_n))\mu(A_n)\\
&=\sum\limits_{n\in N}\Phi_2[\Phi_1^{-1}(\mu(A_n).\Phi_1^{-1}\circ\Phi_1(f(A_n))\frac{u(A_n)}{\Phi_1^{-1}(\mu(A_n))}]\mu(A_n)\\
&\leq b.\sum\limits_{n\in N}\Phi_2\circ\Phi_1^{-1}[c\mu(A_n).\Phi_1(f(A_n))].\Phi_2[\frac{u(A_n)}{\Phi_1^{-1}(\mu(A_n))}]\mu(A_n)\\
&\leq b.\sum\limits_{n\in N}\Phi_2\circ\Phi_1^{-1}[c\mu(A_n).\Phi_1(f(A_n))].M\\
&\leq bM.\Phi_2\circ\Phi_1^{-1}[c\sum\limits_{n\in N}\mu(A_n).\Phi_1(f(A_n))].
\end{align*}
Since $\parallel f\parallel_{\Phi_1}\leq1$ ,Therefor we get that
\begin{align*}
\int_X\Phi_2(M_uf)d\mu&\leq bM.\Phi_2\circ\Phi_1^{-1}(c)<\infty.\\
\end{align*}
This implies that $\|M_u(f)\|_{\Phi_2}\leq bM.\Phi_2\circ\Phi_1^{-1}(c)+1$ and so $M_u$ is bounded.\\
\end{proof}

In the sequel we give some necessary and sufficient conditions under which the composition operator $C_T$ is a bounded operator between different Orlicz space.

\begin{thm}\label{t4} Let $T:X\rightarrow X$ be a non-singular measurable transformation and $\Phi_1, \Phi_2$ be Young's functions such that $\Phi_2\nless\Phi_1$. If T induces the composition operator $C_T:L^{\Phi_1}(\Sigma)\rightarrow L^{\Phi_2}(\Sigma)$, then
\begin{enumerate}
\item[(i)] $f_0(x)=0$ for $\mu$-almost all $x\in B$.

\item[(ii)] $\sup\limits_{n\in N}\frac{\Phi_1^{-1}(\frac{1}{\mu(A_n)})}{\Phi_2^{-1}(\frac{1}{f_0(A_n)\mu(A_n)})}<\infty$.
\end{enumerate}
\end{thm}

\begin{proof} If $\mu(\{x\in B; f_0(x)\neq0\})>0$, then there exists a positive constant $\delta$ such that $\mu(\{x\in B; f_0(x)>\delta\})>0$. Let $E=\{x\in B; f_0(x)>\delta\}$. Since $\mu(E)>0$ and $E$ is non-atomic, by Lemma \ref{l2} and Proposition \ref{p1}, there exists $f\in L^{\Phi_1}(\Sigma)$ such that $f\notin L^{\Phi_2}(E)$. Then we have
\begin{align*}
\infty&=\int_E\delta\Phi_2(\frac{f(x)}{\parallel C_Tf\parallel_{\Phi_2}})d\mu\\
&\leq\int_Ef_0(x)\Phi_2(\frac{ f(x)}{\parallel C_Tf\parallel_{\Phi_2}})d\mu\\
&\leq\int_X\Phi_2(\frac{C_Tf(x)}{\parallel C_Tf\parallel_{\Phi_2}})d\mu\\
&\leq1,
\end{align*}
which is a contraction. Now we prove (ii), for this we set $f_n=\Phi_1^{-1}(\frac{1}{\mu(A_n)})\chi_{A_n}$. It is clear that $f_n\in L^{\Phi_1}(\Sigma)$ and $\parallel f_n\parallel_{\Phi_1}=1$. Hence we have
\begin{align*}
1&\geq\int_X\Phi_2(\frac{ C_Tf_n(x)}{\parallel C_Tf_n\parallel_{\Phi_2}})d\mu\\
&=\int_{A_n}f_0(x)\Phi_2(\frac{\Phi_1^{-1}(\frac{1}{\mu(A_n)})}{\parallel C_Tf_n\parallel_{\Phi_2}})d\mu\\
&=f_0(A_n)\Phi_2(\frac{\Phi_1^{-1}(\frac{1}{\mu(A_n)})}{\parallel C_Tf_n\parallel_{\Phi_2}})\mu(A_n).
\end{align*}
Then we get that $\frac{\Phi_1^{-1}(\frac{1}{\mu(A_n)})}{\parallel C_Tf_n\parallel_{\Phi_2}}\leq\Phi_2^{-1}(\frac{1}{f_0(A_n)\mu(A_n)})$,
this implies that
\begin{center}
$\sup\limits_{n\in N}\frac{\Phi_1^{-1}(\frac{1}{\mu(A_n)})}{\Phi_2^{-1}(\frac{1}{f_0(A_n)\mu(A_n)})}\leq\|C_T\|<\infty$
\end{center}
and so (ii) holds.
\end{proof}

\begin{thm}\label{t45}  Let $T:X\rightarrow X$ be a non-singular measurable transformation, $\Phi_1,\Phi_2\in\triangle'$.  and $\Phi_2\circ \Phi^{-1}_1$ be a Young's function. Then $T$ induces a the composition operator $C_T:L^{\Phi_1}(X)\rightarrow L^{\Phi_2}(X)$, if
\begin{enumerate}

\item[(i)] $f_0(x)=0$ for $\mu$-almost all $x\in B$ .

\item[(ii)] $\sup\limits_{n\in N}\Phi_2[\frac{1}{\Phi_1^{-1}(\mu(A_n))}]f_0(A_n)\mu(A_n)<\infty$.
\end{enumerate}
\end{thm}

\begin{proof}
 Suppose that (i) and (ii) hold. Put $\sup\limits_{n\in N}\Phi_2[\frac{1}{\Phi_1^{-1}(\mu(A_n)})]f_0(A_n)\mu(A_n)=M$. Then for each $f\in L^{\Phi_1}(X)$, we have
\begin{align*}
\int_X\Phi_2(C_Tf)d\mu&=\int_{x\in B}f_0(x)\Phi_2(f(x))d\mu+\int_{x\in \bigcup A_n}f_0(x)\Phi_2(f(x))d\mu\\
&=\sum\limits_{n\in N}\int_{x\in A_n}f_o(x)\Phi_2(f(x))d\mu\\
&=\sum\limits_{n\in N}f_o(A_n)\Phi_2(f(A_n))\mu(A_n)\\
&=\sum\limits_{n\in N}f_o(A_n)\Phi_2[\Phi_1^{-1}(\mu(A_n).\Phi_1^{-1}\circ\Phi_1(f(A_n))\frac{1}{\Phi_1^{-1}(\mu(A_n))}]\mu(A_n)\\
&\leq b.\sum\limits_{n\in N}f_o(A_n)\Phi_2\circ\Phi_1^{-1}[c\mu(A_n).\Phi_1(f(A_n))].\Phi_2[\frac{1}{\Phi_1^{-1}(\mu(A_n))}]\mu(A_n)\\
&\leq b.\sum\limits_{n\in N}\Phi_2\circ\Phi_1^{-1}[c\mu(A_n).\Phi_1(f(A_n))].M\\
&\leq bM.\Phi_2\circ\Phi_1^{-1}[c\sum\limits_{n\in N}\mu(A_n).\Phi_1(f(A_n))].
\end{align*}
Since $\parallel f\parallel_{\Phi_1}\leq1$ , then we get that
\begin{align*}
\int_X\Phi_2(C_Tf)d\mu&\leq bM.\Phi_2\circ\Phi_1^{-1}(c)<\infty\\
\end{align*}
This implies that $\|C_T(f)\|_{\Phi_2}\leq (bM.\Phi_2\circ\Phi_1^{-1}(c)+1)$ and so $C_T$ is bounded.\\
\end{proof}

\begin{thm}\label{t46} Let $T:X\rightarrow X$ be a  non-singular measurable transformation, $\Phi_2\in\triangle'$ and $\Phi_1(xy)\leq\Phi_2(x)+\Phi_3(y)$ and
\begin{enumerate}

\item[(i)] $T$ induces a the composition operator $C_T:L^{\Phi_1}(\Sigma)\rightarrow L^{\Phi_2}(\Sigma)$.

\item[(ii)]  $\mu T^{-1}(B)=0$ and there is a constant M such that $\Phi_1^{-1}(\frac{1}{\mu(A_n)})\leq M\Phi_2^{-1}(\frac{1}{\mu T^{-1}(A_n)})$.

\item[(iii)] $f_0(x)=0$ for $\mu$-almost all $x\in B$ and  $\sup\limits_{n\in N} f_0(A_n).\Phi_2\Phi_3^{-1}(\frac{1}{\mu(A_n)})<\infty$.
Then $i\Rightarrow ii,iii$ and $ii\Rightarrow iii$
\end{enumerate}
\end{thm}

\begin{proof}
(i)$\Rightarrow (ii,iii)$. Since $C_T$ is a composition operator, then for every $n\in \mathbb{N}$ and $y\geq0$ there exists $K,M>0$ such that
\begin{align*}
\int_{A_n} \Phi_2(y)f_0(A_n)-\Phi_1(Ky)d\mu<M&\Longrightarrow \forall n\in \mathbb{N}, [\Phi_2(y)f_0(A_n)-\Phi_1(Ky)]\mu(A_n)<M\\
&\Longrightarrow \forall n\in \mathbb{N} , y\geq0, \Phi_2(y)f_0(A_n)<M\\
&\Longrightarrow \forall n\in \mathbb{N}, \Phi_2(\Phi_3^{-1}(\frac{1}{\mu(A_n)}))f_0(A_n)<M.
\end{align*}
By Lemma \ref{l2} and Theorem \ref{t4} we get that $f_0(x)=0$ for $\mu$-almost all $x\in B$, so we have (iii). For the implication (i)$\Rightarrow (ii)$, since $C_T$ is  bounded, then for a $M'$ and all $n\in \mathbb{N}$, we have
$$\parallel C_T(\chi_{A_n})\parallel_{\Phi_2} \leq \parallel C_T\parallel. \parallel\chi_{A_n}\parallel_{\Phi_1}$$
and so
 $$\Phi_1^{-1}(\frac{1}{\mu(A_n)})\leq M'\Phi_2^{-1}(\frac{1}{\mu T^{-1}(A_n)}).$$
Again, by Lemma \ref{l2} and Theorem \ref{t4}, we have $f_0(x)=0$ for $\mu$-almost all $x\in B$ and so $\mu T^{-1}(B)=0$. Finally we show that (ii)$\Rightarrow (iii)$. Let $n\in \mathbb{N}$, then
\begin{align*}
\Phi_1^{-1}(\frac{1}{\mu(A_n)})&\leq M'\Phi_2^{-1}(\frac{1}{\mu T^{-1}(A_n)})\\
&=M'\Phi_2^{-1}(\frac{1}{f_0(A_n)\mu(A_n)})\\
&\leq M'\frac{\Phi_2^{-1}(\frac{1}{\mu(A_n)})}{\Phi_2^{-1}(\frac{f_0(A_n)}{b})}.
\end{align*}
Therefore
$$\Phi_2^{-1}(\frac{f_0(A_n)}{b}).\Phi_3^{-1}(\frac{1}{\mu(A_n)})<2M',$$
and so
$$\sup\limits_{n\in N} \Phi_2^{-1}(f_0(A_n)).\Phi_3^{-1}(\frac{1}{\mu(A_n)})<\infty.$$
 Hence by basic analysis information we get that $\sup\limits_{n\in N}\Phi_2^{-1}(f_0(A_n)<\infty $ and $\sup\limits_{n\in N}\Phi_3^{-1}(\frac{1}{\mu(A_n)})<\infty$. Finally we conclude that $$\sup\limits_{n\in N} f_0(A_n).\Phi_2(\Phi_3^{-1}(\frac{1}{\mu(A_n)}))<\infty.$$
This completes the proof.
\end{proof}

\begin{lem}\label{l3} Let $\Phi_1, \Phi_2$ be Young's functions and $\Phi_2\in \bigtriangledown'$, then for every $f\in \mathcal{D}(M_{\Phi^{-1}(f_0)})\subseteq L^{\Phi_1}(\Sigma)$ we get that $$\|C_T(f)\|_{\Phi_2}\leq b\|M_{\Phi_2^{-1}(f_0)}f\|_{\Phi_2}.$$
\end{lem}

\begin{proof} Let $f\in \mathcal{D}(M_{\Phi_2^{-1}(f_0)})\subseteq L^{\Phi_1}(\Sigma)$, the by definition of $\|.\|_{\Phi_2}$ we have
\begin{align*}
\|C_T(f)\|_{\Phi_2}&=\inf\{k:\int_X \Phi_2(\frac{f(T(x))}{k})d\mu\leq1\}\\
&=\inf\{k:\int_X f_0(x)\phi_2(\frac{f(x)}{k})d\mu\leq1\}\\
&=\inf\{k:\int_X \Phi_2(\Phi_2^{-1}(f_0(x)))\Phi_2(\frac{f(x)}{k})d\mu\leq1\}\\
&\leq \inf\{k:\int_X \Phi_2(\frac{b\Phi_2^{-1}(f_0(x))f(x)}{k})d\mu\leq1\}\\
&\leq b\inf\{k/b:\int_X \Phi_2(\frac{\Phi_2^{-1}(f_0(x))f(x)}{k/b})d\mu\leq1\}\\
&=b\|M_{\Phi_2^{-1}(f_0)}f\|.
\end{align*}
So we have $$\|C_T(f)\|_{\Phi_2}\leq b\|M_{\Phi_2^{-1}(f_0)}f\|_{\Phi_2}.$$
\end{proof}

\begin{lem}\label{l4} Let $\Phi_1. \Phi_2$ be Young's functions and $\Phi_2\in \bigtriangleup'$, then for every $f\in \mathcal{D}(C_T)\subseteq L^{\Phi_1}(\Sigma)$ we have
 $$\|M_{\Phi_2^{-1}(f_0)}f\|\leq c\|C_T(f)\|_{\Phi_2}.$$
 \end{lem}

 \begin{proof} Let $f\in \mathcal{D}(C_T)\subseteq L^{\Phi_1}(\Sigma)$ and $c\geq 1$ for the definition, $\Phi_2(xy)\leq c\Phi_2(x)\Phi_2(y)$ , then we have;
\begin{align*}
\|M_{\Phi_2^{-1}(f_0)}f\|&=\inf\{k:\int_X \Phi_2(\frac{\Phi_2^{-1}(f_0(x))f(x)}{k})d\mu\leq1\}\\
&\leq \inf\{k:\int_X c\Phi_2(\Phi_2^{-1}(f_0(x)))\Phi_2(\frac{f(x)}{k})d\mu\leq1\}\\
&\leq \inf\{k:\int_X cf_0(x))\Phi_2(\frac{f(x)}{k})d\mu\leq1\}\\
&=\inf\{k:\int_X c\Phi_2(\frac{(f\circ T)(x)}{k})d\mu\leq1\}\\
&=c\inf\{k/c:\int_X \Phi_2(\frac{(f\circ T)(x)}{k/c})d\mu\leq1\}\\
&=c\|C_T(f)\|_{\Phi_2}.
\end{align*}
Hence the proof is completed.
\end{proof}
Here we recall a definition that came in \cite{tak}. For any $F\in \Sigma$, we put
$$Q_T(F)=\inf\{b\geq0:\mu\circ T^{-1}(E)\leq b\mu(E)\quad   (E\in \Sigma)\}.$$
Then we have the following two lemmas.

\begin{lem}\label{l55}
\cite{tak} For any $F\in \Sigma$, we have $Q_T(F)=ess.sup_{x\in F}f_0(x)$.
\end{lem}

\begin{lem}\label{l5} Let $\Phi_2, \Phi_3$ be Young's functions, then we have

$$\displaystyle\int_X\Phi_3\circ\Phi_2^{-1}(f_0)d\mu=\inf\{\sum^{\infty}_{j=1}\Phi_3\circ\Phi_2^{-1}(Q_T(F_j))\mu
(F_j);\{F_j\}\in \mathcal{P}_X\},$$
where $\mathcal{P}_X$ is the set of all partitions of $X$.
\end{lem}

\begin{proof} Let $\displaystyle I=\inf\{\sum^{\infty}_{j=1}\Phi_3\circ\Phi_2^{-1}(Q_T(F_j))\mu(F_j):\{F_j\}\in \mathcal{P}_X\}$. For the partition $\{F_j\}$ of $X$, by using Lemma \ref{l55} we have
\begin{align*}
\displaystyle\int_X\Phi_3\circ\Phi_2^{-1}(f_0(x))d\mu&=\sum^{\infty}_{j=1}\displaystyle\int_{F_j}\Phi_3\circ\Phi_2^{-1}
(f_0(x))d\mu\\
&\leq\sum^{\infty}_{j=1}\Phi_3\circ\Phi_2^{-1}(ess.sup_{x\in F_j} f_0(x))\mu(F_j)\\
&=\sum^{\infty}_{j=1}\phi_3\circ\phi_2^{-1}(Q_T(F_j))\mu(F_j).
\end{align*}
Then we get that $$\int_X\Phi_3\circ\Phi_2^{-1}(f_0(x))d\mu\leq I.$$

Conversely; let $a>1$ be arbitrarily and set
$$G_m=\{x\in X; a^{m-1}\leq\Phi_3\circ\Phi_2^{-1}(f_0)<a^m\}$$ for each integer m. If $\{F_j\}_{j=1}^\infty$ is a rearrangement of $\{G_j\}_{m=-\infty}^\infty$ and $\{x\in X: f_0(x)=0\}$, then $\{F_j\}_{j=1}^\infty$  clearly becomes a partition of X. Therefore by the Lemma \ref{l55} we have
\begin{align*}
I&\leq\sum^{\infty}_{j=1}\Phi_3\circ\Phi_2^{-1}(Q_T(F_j))\mu(F_j)\\
&=\sum^{\infty}_{j=1}\Phi_3\circ\Phi_2^{-1}(ess.sup_{x\in F_j}f_0(x))\mu(F_j)\\
&=\sum^{\infty}_{m=-\infty}\Phi_3\circ\Phi_2^{-1}(ess.sup_{x\in F_j}f_0(x))\mu(G_m)\\
&\leq\sum^{\infty}_{m=-\infty}a^m\mu(G_m)\\
&=a\sum^{\infty}_{m=-\infty} a^{m-1}\mu(G_m)\\
&\leq a\sum^{\infty}_{m=-\infty}\displaystyle\int_{G_m}\Phi_3\circ\Phi_2^{-1}(f_0(x))d\mu\\
&=a\sum^{\infty}_{j=1}\displaystyle\int_{F_j}\Phi_3\circ\Phi_2^{-1}(f_0(x))d\mu\\
&=a\int_X\Phi_3\circ\Phi_2^{-1}(f_0(x))d\mu.
\end{align*}
Since this inequality holds for any $a>1$, then proof is completed.
\end{proof}

By using the Theorem \ref{t2} and Lemma \ref{l4} we give a necessary condition for  boundedness of the composition operator $C_T$.

\begin{thm}\label{t47}
Let $\Phi_2\in\triangle'$ and $T$ induces a composition operator $C_T:L^{\Phi_1}(\Sigma)\rightarrow L^{\Phi_2}(\Sigma)$, then there exists a Young function $\Phi_3$ such that
$f_0\in L^{\Psi_3\circ\Psi_1\circ\Phi_2^{-1}}(\Sigma)$.
\end{thm}
\begin{proof} (i) Since $C_T$ is composition operator, by Lemma \ref{l4} we get that
$$M_{\Phi_2^{-1}(f_0)}:L^{\Phi_1}(\Sigma)\rightarrow L^{\Phi_2}(\Sigma)$$
is multiplication operator. Therefore by Theorem \ref{t2}, for a $\Phi_3$, we have ${\Phi_2^{-1}(f_0)}\in L^{\Psi_3\circ\Psi_1}$ and so
$$ \int_X\Psi_3\circ\Psi_1(\Phi_2^{-1}(f_0))d\mu<\infty$$
that implies that $f_0\in L^{\Psi_3\circ\Psi_1\circ\Phi_2^{-1}}(\Sigma)$.
\end{proof}

\begin{thm}\label{t48}
$f_0\in L^{\Phi_3\circ\Phi_2^{-1}}(\Sigma)$ if and only if there exists a partition $\{F_j\}_{j=1}^\infty$ of $\Sigma$ such that $\sum^{\infty}_{j=1}\Phi_3\circ\Phi_2^{-1}(Q_T(F_j))\mu(F_j)<\infty$.
\end{thm}
\begin{proof}
By Lemma \ref{l5} it is easy to prove.
\end{proof}

Let $\Phi(x)=\frac{x^p}{p}$ for $x\geq0$, where $1<p<\infty$. It is clear that $\Phi$ is a Young's function and $\Psi(x)=\frac{x^{p'}}{p'}$, where $1<p'<\infty$ and $\frac{1}{p}+\frac{1}{p'}=1$. These observations and Theorems \ref{t3}, \ref{t4}, \ref{t44}, \ref{t45}, \ref{t46}, \ref{t47}, \ref{t48}, give us the next Remark.

\begin{rem}\begin{itemize}\item[(a)] Let $M_{u}:\mathcal{D}(M_u)\subseteq L^{p}(\Sigma)\rightarrow L^{q}(\Sigma)$ be well defined. Then the operator $M_{u}$ from $L^{p}(\Sigma)$ into $L^{q}(\Sigma)$, where $1<p<q<\infty$, is bounded if and only if the followings hold:
\begin{enumerate}
\item[(i)] $u(x)=0$ for $\mu$-almost all $x\in B$.

\item[(ii)] $\sup_{n\in \mathbb{N}}\frac{|u(A_n)|^r}{\mu(A_n)}<\infty$, where $q^{-1}+r^{-1}=p^{-1}$.
\end{enumerate}
\item[(b)] Let $C_{T}:\mathcal{D}(C_T)\subseteq L^{p}(\Sigma)\rightarrow L^{q}(\Sigma)$ be well defined. Then the followings are equivalent:
\begin{enumerate}
\item[(i)] $C_T$ is bounded from $L^p(\Sigma)$ into $L^q(\Sigma)$.

\item[(ii)] $f_0(x)=0$ for $\mu$-almost all $x\in B$ and
$\sup_{n\in \mathbb{N}}\frac{|f_0(A_n)|^p}{\mu(A_n)^{q-p}}<\infty$, where $q^{-1}+r^{-1}=p^{-1}$.

\item[(iii)] $\mu\circ T^{-1}(B)=0$ and there is a constant $k$ such that $\mu\circ T^{-1}(A_n)^p\leq k\mu(A_n)^q$ for all $n\in \mathbb{N}$.
\end{enumerate}
\item[(c)] Let $M_{u}:\mathcal{D}(M_u)\subseteq L^{p}(\Sigma)\rightarrow L^{q}(\Sigma)$ be well defined. Then the operator $M_{u}$ from $L^{p}(\Sigma)$ into $L^{q}(\Sigma)$, where $1<q<p<\infty$, is bounded if and only if $u\in L^r(\Sigma)$, where $p^{-1}+r^{-1}=q^{-1}$.

\item[(d)] Let $C_T:\mathcal{D}(C_T)\subseteq L^{p}(\Sigma)\rightarrow L^{q}(\Sigma)$, where $1<q<p<\infty$, be well defined. Then the followings are equivalent:
\begin{enumerate}
\item[(i)] $C_T$ is a bounded operator from $L^{p}(\Sigma)$ into $L^{q}(\Sigma)$.

\item[(ii)] $f_0\in L^{\frac{r}{q}}(\Sigma)$, where $p^{-1}+r^{-1}=q^{-1}$.

\item[(iii)] There exists a partition $\{F_j\}^{\infty}_{j=1}$ of $X$ such that $\sum^{\infty}_{j=1}Q_T(F_j)^{\frac{r}{q}}\mu(F_j)<\infty$.
\end{enumerate}
\end{itemize}
\end{rem}

\section{ \sc\bf  Closed range multiplication and composition operators }
In this section we are going to investigate closed range multiplication and composition operators between different Orlicz spaces. First we give a fundamental lemma, then we consider the closed range multiplication operator.
\begin{lem}\label{l30}
 Let $(\Phi_i,\Psi_i), \ \ i=1,2$ be two complementary Young's functions pairs such that $\Psi_2\circ\Psi_1^{-1}$ be Young's function. If $\Psi_1\in \Delta'$, then $\Psi_1(xy)\leq\Psi_2(x)+\Psi_3(\Psi_1(y))$,
 for all $x,y\geq0$.
\end{lem}

\begin{proof}
 If we take $\Phi_3=\Psi_2\circ\Psi_1^{-1}$, then

 $$\Psi_3(y)=sup\{xy-\Phi_3(x): \ x\geq 0\}=sup\{xy-\Psi_2\circ\Psi_1^{-1}(x): \ x\geq 0\}.$$
 Hence
 $$\Psi_3(\Psi_1(y))=sup\{x\Psi_1(y)-\Psi_2\circ\Psi_1^{-1}(x): \ x\geq 0\}$$
 and so
 $$\Psi_3(\Psi_1(y))=sup\{\Psi_1(x)\Psi_1(y)-\Psi_2(x): \ x\geq 0\}.$$
 Since $\Psi_1\in \Delta'$, then  we have;
 $$\Psi_3(\Psi_1(y))\geq\Psi_1(xy)-\Psi_2(x),$$
 for all $x,y\geq0$
\end{proof}
Now we characterize closed range multiplication operators $M_u:L^{\Phi_1}(\Sigma)\rightarrow L^{\Phi_2}(\Sigma)$, when $\Phi_2(xy)\leq\Phi_1(x)+\Phi_3(y)$ , for all $x\geq0$.

\begin{thm}\label{t33} Let $\Phi_i, \  i=1,2,3$ be Young's functions such that $\Phi_1\in \Delta'$ and $\Phi_2(xy)\leq\Phi_1(x)+\Phi_3(y)$ , for all $x,y\geq0$. If $u\in L^{\Phi_3}(\Sigma)$, then  $M_u:L^{\Phi_1}(\Sigma)\rightarrow L^{\Phi_2}(\Sigma)$ is a multiplication operator and the following cases are equivalent:
\begin{enumerate}

\item[(a)]$u(x)=0$ for $\mu$-almost all $x\in B$ and the set $E=\{n\in \mathbb{N}: u(A_n)\neq 0\}$ is finite.

\item[(b)]$M_u$ has finite rank.

\item[(c)]$M_u$ has closed range.
\end{enumerate}
\end{thm}

\begin{proof} Let $S=S(u)$. Since $M_u$ is a non-zero operator, then $\mu(S)>0$. First we prove the implication (a)$\Rightarrow$(b). So there exists $r\in \mathbb{N}$ such that
$$S=\bigcup\limits_{n\in E}A_n=A_{n_1}\bigcup...\bigcup A_{n_r}.$$
It is clear that the set $\{\chi_{A_{n_1}},\ldots,\chi_{A_{n_1}}\}$ is a generator of the subspace
\begin{center}
$\{g\in L^{\Phi_2}(X): g(x)=0$  for $\mu-$almost all  $x\in X\setminus S\}\cong L^{\Phi_2}(S)$.
\end{center}
Since $L^{\Phi_2}(S)$ contains $M_u(L^{\Phi_1}(X))$, then $M_u$ has finite rank.

The implication (b)$\Rightarrow$ (c) is obvious. Finally we prove the implication (c)$\Rightarrow$ (a). If $\mu\{x\in B: u(x)\neq0\}>0$, then for some $\delta>0$ we have $\mu\{x\in B: u(x)\geq\delta\}>0$. Let $G=\{x\in B: u(x)\geq\delta\}>0$. It is easy to see that $M_{u\mid_{G}}$ is a multiplication operator from $L^{\Phi_1}(G)$ into $L^{\Phi_2}(G)$. Also, $M_{u\mid_{G}}(L^{\Phi_1}(G))=L^{\Phi_2}(G)$, since for every measurable subset $A$ of $G$ with $\mu(A)<\infty$ and $f_A=\frac{1}{u}\chi_{A}$ we have
$$\int_G\Phi_1(f_A)d\mu=\int_A\Phi_1(\frac{1}{u(x)})d\mu\leq\Phi_1(\frac{1}{\delta})\mu(A)<\infty.$$
  Hence $f_A=\frac{1}{u}\chi_{A}\in L^{\Phi_1}(\Sigma)$ and $M_{u\mid_{G}}(f_A)=\chi_{A}$. This implies that $M_{u\mid_{G}}(L^{\Phi_1}(G))=L^{\Phi_2}(G)$ and so $M_{u\mid_{G}}$ is invertible and its inverse operator is a multiplication operator as follows:
$$M_{\frac{1}{u}}:L^{\Phi_2}(G)\rightarrow L^{\Phi_1}(G),\quad  M_{\frac{1}{u}}(f)=\frac{1}{u}.f.$$
By Theorem \ref{t3}, we have $\frac{1}{u(x)}=0$ for $\mu-$ almost all $x\in G$, which is impossible. This contradiction implies that $u(x)=0$ for $\mu-$ almost all $x\in B$.
Now we show that $E$ is finite. Clearly $S=\bigcup\limits_{n\in E}A_n$ and $E\neq\varnothing$. If we define $M_{\frac{1}{u}}:L^{\Phi_2}(S)\rightarrow L^{\Phi_1}(S)$ once more, similar to the previous case, $M_{\frac{1}{u}}$ is a multiplication operator.

So by Theorem \ref{t3},

 $$\sup\limits_{n\in N} \frac{1}{u(A_n)}.\Phi_3^{-1}(\frac{1}{\mu(A_n)})\leq \infty.$$

Let $C=\sup\limits_{n\in N} \frac{1}{u(A_n)}.\Phi_3^{-1}(\frac{1}{\mu(A_n)})$. It is clear that $C>0$. Since $E\neq\varnothing$, and for all $n\in E$, $1\leq \Phi_3(Cu(A_n)).\mu(A_n)$,  Then we get that
 \begin{align*}
 \sum_{n\in E}1&\leq\sum_{n\in E}\Phi_3(Cu(A_n)).\mu(A_n)\\
 &=\sum_{n\in E}\int_{A_n}\Phi_3(Cu(x))d\mu\\
&\leq\int_{X}\Phi_3(Cu(x))d\mu<\infty.
 \end{align*}
This implies that $E$ should be finite.
\end{proof}

In the next theorem we characterize closed range multiplication operators $M_u:L^{\Phi_1}(\Sigma)\rightarrow L^{\Phi_2}(\Sigma)$, when $\Phi_1(xy)\leq\Phi_2(x)+\Phi_3(y)$ , for all $x\geq0$.

\begin{thm}\label{t34} Let $\Phi_1, \Phi_2, \Phi_3$ be Young's functions such that $\Phi_1(xy)\leq\Phi_2(x)+\Phi_3(y)$ for all $x,y\geq 0$ and $\Phi_2\in\Delta'$. If $M_u:L^{\Phi_1}(\Sigma)\rightarrow L^{\Phi_2}(\Sigma)$ is a multiplication operator and  $\frac{1}{u}\in L^{\Phi_3}(\varSigma)$,  then the following cases are equivalent:
\begin{enumerate}

\item[(a)]the set $E=\{n\in \mathbb{N}: u(A_n)\neq 0\}$ is finite.

\item[(b)]$M_u$ has finite rank.

\item[(c)]$M_u$ has closed range.
\end{enumerate}
\end{thm}

\begin{proof} By Theorem \ref{t3} we have $u(x)=0$ for $\mu$-almost all $x\in B$. The implications (a)$\Rightarrow$ (b) and (b)$\Rightarrow$ (c) is simillar to Theorem \ref{t33}. Now we prove the implication (c)$\Rightarrow$ (a). Let $S=\bigcup\limits_{n\in E}A_n$ and $E\neq\varnothing$. Since  $M_u:L^{\Phi_1}(S)\rightarrow L^{\Phi_2}(S)$ is bounded, by Theorem \ref{t3}, we have;
	$$\sup\limits_{n\in N} u(A_n).\Phi_3^{-1}(\frac{1}{\mu(A_n)})\leq \infty$$
	
	Let $C=\sup\limits_{n\in N} u(A_n).\Phi_3^{-1}(\frac{1}{\mu(A_n)})$. It is clear that $C>0$, since $E\neq\varnothing$, and for all $n\in E$, $1\leq \Phi_3(\frac{C}{u(A_n)}).\mu(A_n)$. Therefor we can write;
	\begin{align*}
	\sum_{n\in E}1&\leq\sum_{n\in E}\Phi_3(\frac{C}{u(A_n)}).\mu(A_n)\\
	&=\sum_{n\in E}\int_{A_n}\Phi_3(\frac{C}{u(A_n)})d\mu\\
	&\leq\int_{X}\Phi_3(\frac{C}{u(A_n)})d\mu<\infty.
	\end{align*}
	This means that $E$ should be finite.
\end{proof}

 Here we begin to investigate closed range composition operators between different Orlicz spaces. First we give an elementary lemma.

 \begin{lem}\label{l31}Let $\Phi_1, \Phi_2$  be Young's functions and $T$ be a  non-singular measurable transformation on $X$ such that  $C_T:L^{\phi_1}(X)\rightarrow L^{\phi_2}(X)$ is a composition operator.If $T$ is surjective, then $C_T$ is injective.
\end{lem}
 \begin{proof}
 It is easy to prove.
 \end{proof}

 Now we characterize closed range composition operators $C_T:L^{\Phi_1}(\Sigma)\rightarrow L^{\Phi_2}(\Sigma)$, when $\Phi_2(xy)\leq\Phi_1(x)+\Phi_3(y)$ , for all $x\geq0$.
\begin{thm}\label{t35} Let $\Phi_1, \Phi_2, \Phi_3$ be Young's functions such that $\Phi_2(xy)\leq\Phi_1(x)+\Phi_3(y)$ for all $x,y\geq0$ and $T$ be a surjective non-singular measurable transformation on $X$. If $C_T: L^{\Phi_1}(\Sigma)\rightarrow L^{\Phi_2}(\Sigma)$ is a composition operator, then the following cases are equivalent:
\begin{enumerate}
\item[(a)] $C_T$ has closed range.

\item[(b)] $f_0(x)=0$ for $\mu$-almost all $x\in B$, and the set $\{n\in \mathbb{N}: f_0(A_n)\neq 0\}$ is finite.

\item[(c)] $\mu T^{-1}(B)=0$, and the set $E_T=\{n\in \mathbb{N}: \mu T^{-1}(A_n)\neq 0\}$ is finite.

\item[(d)] $C_T$ has finite rank.
\end{enumerate}
\end{thm}

\begin{proof} The implications (b)$\Rightarrow$ (c) and (d)$\Rightarrow$ (a) are obvious. First we prove the implication (a)$\Rightarrow$ (b). Since T is surjective, By Lemma \ref{l31} $C_T$ is injective and if $C_T$ has closed range, then by the Lemma \ref{l3} we get that $M_{\Phi_2^{-1}(f_0)}$  has
closed range. Hence by Theorem \ref{t3} we have $\Phi_2^{-1}(f_0)(x)=0$ for $\mu$-almost all $x\in B$ and
$\{n\in \mathbb{N}: \Phi_2^{-1}(f_0)(A_n)\neq 0\}$ is finite. Hence $f_0(x)=0$ for $\mu$-almost all $x\in B$ and the set $\{n\in \mathbb{N}: f_0(A_n)\neq 0\}$ is finite.

Finally we show that the implication (c)$\Rightarrow$ (d) holds. Suppose (c) holds, then it is easy to show that $C_T(L^{\Phi_1}(\Sigma))$ is contained in the subspace generated by $\{\chi_{T^{-1}(A_n)}\}_{n\in E_T}$. Since $E_T$ is finite, then $C_T(L^{\Phi_1}(\Sigma))$ is finite dimensional and so $C_T$ has finite rank.
\end{proof}

\begin{cor} If $X$ is non-atomic, under the assumptions of Theorem \ref{t35}, there is not any non-zero closed range composition operator from $L^{\Phi_{1}}(\Sigma)$ into $L^{\Phi_{2}}(\Sigma)$.
\end{cor}
In the next theorem we characterize closed range composition operators $C_T:L^{\Phi_1}(\Sigma)\rightarrow L^{\Phi_2}(\Sigma)$, when $\Phi_1(xy)\leq\Phi_2(x)+\Phi_3(y)$ , for all $x,y\geq0$.
\begin{thm}\label{t36}  Let $\Phi_1, \Phi_2, \Phi_3$ be Young's functions such that $\Phi_2\in \nabla'\bigcap \bigtriangleup_{2}$ and $\Phi_1(xy)\leq\Phi_2(x)+\Phi_3(y)$, for all $x\geq0$. If $T$ is a non-singular measurable transformation on $X$ and $C_T: L^{\Phi_1}(\Sigma)\rightarrow L^{\Phi_2}(\Sigma)$ is a composition operator, then the followings are equivalent:
\begin{enumerate}

\item[(a)] $C_T$ has closed range.

\item[(b)] The set $\{n\in \mathbb{N}: f_0(A_n)\neq 0\}$ is finite.

\item[(c)]The set $\{n\in \mathbb{N}: \mu T^{-1}(A_n)\neq 0\}$ is finite.

\item[(d)] $C_T$ has finite rank.
\end{enumerate}
\end{thm}

\begin{proof}
By using Lemma \ref{l3}, Theorem \ref{t34} and similar method of Theorem \ref{t35}, we get proof.
\end{proof}

In the net remark we derive characterizations of bounded and closed range multiplication and composition operators from our main results.

\begin{rem}\begin{itemize}
\item[(1)] Let multiplication operator $M_{u}$ from $L^{p}(\Sigma)$ into $L^{q}(\Sigma)$, where $1<p<q<\infty$, be bounded, then the followings are equivalent:

\begin{enumerate}
\item[(a)] $M_u$ has closed range.

\item[(b)] $M_u$ has finite rank.

\item[(c)] The set $\{n\in \mathbb{N}: u(A_n)\neq0\}$ is finite.

\end{enumerate}

\item[(2)] If $1<q<p<\infty$, then for multiplication operator $M_{u}$ from $L^{p}(\Sigma)$ into $L^{q}(\Sigma)$ the followings are equivalent:

\begin{enumerate}
\item[(a)]  $M_u$ has closed range.

\item[(b)] $M_u$ has finite rank.

\item[(c)]  $u(x)=0$ for $\mu$-almost all $x\in B$, and the set $\{n\in \mathbb{N}: u(A_n)\neq0\}$ is finite.\\
\end{enumerate}
\item[(3)] Let $C_{T}: L^{p}(\Sigma)\rightarrow L^{q}(\Sigma)$, where $1<p<q<\infty$, be bounded. Then the followings are equivalent:
\begin{enumerate}
\item[(a)] $C_T$ has closed range.

\item[(b)] $C_T$ has finite rank.

\item[(c)] The set $\{n\in \mathbb{N}: f_0(A_n)\neq0\}$ is finite.

\item[(d)] The set $\{n\in \mathbb{N}: \mu\circ T^{-1}(A_n)\neq0\}$ is finite.\\
\end{enumerate}
\item[(4)] If $1<q<p<\infty$, then for composition operator $C_T$ from $L^{p}(\Sigma)$ into $L^{q}(\Sigma)$ the followings are equivalent:
\begin{enumerate}
\item[(a)] $C_T$ has closed range.

\item[(b)] $C_T$ has finite rank.

\item[(c)] $f_0(x)=0$ for $\mu$-almost all $x\in B$, and the set $\{n\in \mathbb{N}: f_0(A_n)\neq0\}$ is finite.

\item[(d)] $\mu\circ T^{-1}(B)=0$, and the set $\{n\in \mathbb{N}: \mu\circ T^{-1}(A_n)\neq0\}$ is finite.
\end{enumerate}
\end{itemize}
\end{rem}
Finally we provide some examples to illustrate our main results.
\begin{exam}
If $\Phi$ and $\Psi$ are complementary Young's functions. Since $\frac{|x|^p}{p}\leq\frac{1}{p}(\Phi(x)+\Psi(x^{p-1}))$ and $\frac{|x|^p}{p}\leq\frac{1}{p}(\Phi(x^{p-1})+\Psi(x))$, for $x\geq0$ and $p>2$, then by Theorem \ref{t1} the operators $M_u:L^{\Phi}(\Sigma)\rightarrow L^p(\Sigma)$ and $M_v:L^{\Psi}(\Sigma)\rightarrow L^p(\Sigma)$ are bounded for every $u\in L^{\Psi}$  and $v\in L^{\Phi}$ respectively.

\end{exam}

\begin{exam}
Let $X=[a,b]  \ a,b>1$, $p>1$ and $\mu$ be the Lebesque measure. If we take $\Phi_1(x)=e^{x^p}-x^p-1$, $\Phi_2(x)=\frac{x^p}{p}$ and $\Phi_3(x)=(1+x^p)log(1+x^p)-x^p$. Then easily we get that $\Phi_2(xy)\leq\Phi_1(x)+\Phi_3(y)$. Let $u(x)=\sqrt[p]{x^p-1}$. It is clear that $\int_X\Phi_3(u(x))d\mu<\infty$. So  by Theorem \ref{t1}, $M_u$ is a bounded operator from $L^{\Phi_1}(\Sigma)$ into $L^{\Phi_2}(\Sigma)$.
\end{exam}

\begin{exam}
Suppose $A=(0,a]$, $B=\{lnx: x\in N ,\   x>a\}, \ X=A\cup B,\ \Phi(x)=e^x-x-1,\ \Psi(x)=(1+x)log(1+x)-x$ and for every $C\subseteq X$, $\mu(C)=\mu_1(C\cap A])+\mu_2(C\cap B)$ such that $\mu_1$ is lebesgue measure and $\mu_2(\{lnx\})=\frac{1}{x^3}$ for $lnx\in B$. If we take $u(x)=\frac{1}{x^2}$, then $M_u$ is not bounded from $L^\Phi(X)$ into $L^\Psi(X)$. Because of for $f(x)=x$ we have:
\begin{align*}
\int_X\Phi(f(x))d\mu&=\int_Xe^x-x-1d\mu\\
&=\int_Ae^x-x-1d\mu+\int_Be^x-x-1d\mu<\infty.
\end{align*}
Since
\begin{align*}
\int_Be^x-x-1d\mu&=\sum\limits_{n>a}(e^{lnn}-lnn-1)(\frac{1}{n^3})\\
&<\sum\limits_{n>a}\frac{1}{n^2}<\infty
\end{align*}
But
\begin{align*}
\int_X\Psi(M_u(f(x)))d\mu&=\int_X\Psi(\frac{1}{x})d\mu\\
&=\int_X(1+\frac{1}{x})log(1+\frac{1}{x})-\frac{1}{x}d\mu \\
&>\int_A log(1+\frac{1}{x})-\frac{1}{x}d\mu+\int_BPAR
 log(1+\frac{1}{x})-\frac{1}{x}d\mu\\
&=xlog(1+\frac{1}{x})+ln(1+\frac{1}{x})\mid_0^{a}+
\sum\limits_{n>a}lnn.log(1+\frac{1}{lnn})+ln(1+\frac{1}{lnn})\\
&=\infty
\end{align*}
Thus by Theorem \ref{t1}, we can conclude that for every Young's function $\Phi'$ such that $\Psi(xy)\leq\Phi(x)+\Phi'(y)$, then $u(x)\notin L^{\Phi'}(\Sigma).$\\
Also by Proposition \ref{p0}, there is no non-zero operator $M_u$ from $L^\Phi(\varSigma)$ into  $L^\Psi(\varSigma)$, because of $\Phi(x)<\Psi(x)$ for $x\geq 0.$
\end{exam}

\begin{exam}
Suppose $A=[1,a],B=\{n\in N; a<n\leq 10a\}, \ \Phi(x)=e^x-x-1,\ X=A\cap B, \ \Psi(x)=(1+x)log(1+x)$ and for every $C\subseteq X$, $\mu(C)=\mu_1(C\cap A])+\mu_2(C\cap B)$ such that $\mu_1$ is lebesgue measure and $\mu_2(\{n\})=1,\ n\in B$. If we take $u(x)=\frac{1}{x^p}, \ p>1$, then $M_u$ is  bounded from $L^\Psi(\varSigma)$ into $L^\Phi(\varSigma)$. Because of if $f(x)\in L^{\psi}(\varSigma)$,then;
$$\int_X\Psi(f)d\mu=\int_X(1+f(x))log(1+f(x))d\mu<\infty,$$
Since $\Phi(x)<\Psi(x)$ for $ x\geq 0$, therefore;
\begin{align*}
\int_X\Phi(u(x).f(x))d\mu&=\int_Xe^{u(x).f(x)}-u(x)f(x)-1d\mu\\
&<\int_X(1+u(x).f(x))log(1+u(x).f(x))d\mu\\
&=\int_A(1+f(x))log(1+f(x))d\mu+\int_B(1+f(x))log(1+f(x))d\mu\\
&<\infty.
\end{align*}
Also by Theorem \ref{t33}, $M_u$ has not closed range, since $\mu\{x\in A; u(x)\neq0 \}\neq 0.$
But if we take  $u(x)=0,\ x\in X\cap Q^c$ and $u(x)=\frac{1}{x^p}, x\in X\cap Q$, then by Theorem \ref{t33}, $M_u$ has closed range.
\end{exam}





\begin{thebibliography}{99}
\bibitem{c}
Y. Cui, H. Hudzik, R. Kumar and L. Maligranda,  Composition
operators in Orlicz spaces, \textit{J. Aust. Math. Soc.} {\bf 76} (2004),
189-206.


\bibitem{skn} S. Gupta, B. S Komal and N. Suri, Weighted
composition operators on Orlicz spaces, \textit{Int. J. Contemp. Math.
Sciences. }{\bf 1}, 11-20 (2010).


\bibitem{ho}
William E. Hornor and James E. Jamison, Properties of
isometry-inducing maps of the unit disc, \textit{Complex Variables Theory
Appl.} {\bf 38} (1999), 69-84.


\bibitem{ks} B.S.  Komal  AND  S. Gupta, Multiplication operators between Orlicz
spaces, \textit{Integral equation and operator theory.} {\bf41} (2001),
324-330.

\bibitem{ku} R. Kumar, ‘Composition operators on Orlicz spaces’, Integral Equations Operator Theory
{\bf29} (1997), 17–22.

\bibitem{ra} M. M. Rao, ‘Convolutions of vector fields–II: random walk models’,\textit{ Nonlinear Anal, Theory
Methods Appl.} {\bf47} (2001), 3599-3615.

\bibitem{raor} M.M. Rao, Z.D. Ren, \textit{Theory of Orlicz spaces}, Marcel Dekker,
New York, 1991.


 \bibitem{sima} R. K. Singh, J. S. Manhas, \textit{Composition operators on function spaces}, North-Holland Mathematics Studies, 179,  North-Holland Publishing Co., Amsterdam, 1993.

\bibitem{tak}
H. Takagi and K. Yokouchi, Multiplication and composition
operators between two $L^p$-spaces, \textit{Contemporary Math.} {\bf
232}(1999), 321-338.

\bibitem{z}
A. C. Zaanen, \textit{Integration}, 2nd ed., North-Holland, Amsterdam,
1967.
\end{thebibliography}
\end{document}